\documentclass[reqno,11pt]{amsart}

\usepackage{amsmath,amssymb,amscd,amsthm,latexsym,graphics,enumerate,engord,hyperref,txfonts}
\usepackage[all]{xy}

\usepackage{amsthm,amsmath}
\usepackage[T1]{fontenc}

\usepackage[english]{babel}
\newtheorem{theorem}{Theorem}[section]
\newtheorem{example}{Example}
\newtheorem{remark}{Remark}
\newtheorem{proposition}{Proposition}[section]
\newtheorem{lemma}{Lemma}[section]

\newtheorem{corollary}{Corollary}[section]

\def\cat{{\rm{cat}\hskip1pt}}

\def\TC{{\rm{TC}\hskip1pt}}

\def\HTC{{\rm{HTC}\hskip1pt}}

\def\MTC{{\rm{MTC}\hskip1pt}}

\def\cl{{\rm{cl}\hskip1pt}}

\def\zcl{{\rm{zcl}\hskip1pt}}

\begin{document}
\title{On the rational topological complexity of coformal elliptic spaces}

\author{Said Hamoun}
\author{Youssef Rami}
\address{My Ismail University of Mekn\`es, Department of Mathematics, B. P. 11 201 Zitoune, Mekn\`es, Morocco.}
\email{s.hamoun@edu.umi.ac.ma}
\email{y.rami@umi.ac.ma}
\author{Lucile Vandembroucq}
\address{Centro de Matem\'atica, Universidade do Minho, Campus de Gualtar, 4710-057 Braga, Portugal.}
\email{lucile@math.uminho.pt}

	\begin{abstract} We establish some upper and lower bounds of the rational topological complexity for certain classes of elliptic spaces. Our techniques permit us in particular to show that the rational topological complexity coincides with the dimension of the rational homotopy for some special families of coformal elliptic spaces. 
	\end{abstract}



\maketitle

\section{Introduction}
Let $S$ be a topological space. We recall that the topological complexity $\TC(S)$ of $S$, introduced by Farber \cite{FM}, is defined as the  least integer $m$ such that there exists a family of open subsets $U_0,\cdots , U_m$ covering $S\times S$ and of local continuous sections $s_i: U_i \rightarrow S ^{[0,1]}$ of the map $ev_{0,1}: S^{[0,1]}\rightarrow S\times S$, $\lambda \rightarrow (\lambda  (0),\lambda (1))$. The integer $\TC$ is a homotopy invariant which satisfies $\cat(S) \leq \TC(S) \leq 2 \cat(S)$ where  $\cat(S)$ is the Lusternik–Schnirelmann category of $S$. We refer to \cite{F} and \cite{CLOT} for more information on these invariants.

If $S$ is a simply-connected CW-complex of finite type and if $S_{\!{0}}$ is its rationalization, then lower bounds of $\cat(S)$ and  $\TC(S)$ are given respectively by the rational LS-category $\cat_0(S):=\cat(S_{\!{0}})$ and rational topological complexity $\TC_0(S):=\TC(S_{\!{0}})$. As is known, these invariants can be characterized in terms of a Sullivan model $(\Lambda V,d)$ of $S$, see Sections \ref{sec2} and \ref{sec3} below for more details.  We here recall that, when $S$ is formal (that is, $(\Lambda V,d)\rightarrow (H^*(X; \mathbb{Q}),0)$ is a quasi-isomorphism), we have $\cat_0(S)=\cl_{\mathbb{Q}}(S)$ and $\TC_0(S)=\zcl_{\mathbb{Q}}(S)$ where $\cl_{\mathbb{Q}}(S)$ and $\zcl_{\mathbb{Q}}(S)$ are respectively the cuplength and zero-divisor cuplength of $H^*(S;\mathbb{Q})$. However, in general $\cat_ 0$ and $\TC_0$ can be larger than these cohomological lower bounds.

In this article we study the rational topological complexity of elliptic spaces. Recall that $S$ is elliptic if $\pi_{*}(S) \otimes \mathbb{Q}$ and $ H^*(S; \mathbb{Q})$ are both finite dimensional, see \cite[Ch 6]{FHT} as a general reference. Explicit expressions for the rational LS-category of certain classes of elliptic spaces have been established in \cite{AJ}, \cite{LM02}, \cite{J} and \cite{L}. In the continuity of these works, our general goal is to study the rational topological complexity of an elliptic space in terms of its LS-category. 
Many elliptic spaces, for instance the homogeneous spaces $G/H$, admit a pure minimal Sullivan model $(\Lambda V,d)$, where  pure means $d V^{even}=0$ and $d V^{odd} \subset \Lambda V^{even}$. We will call such a space a pure elliptic space and our work focus on the study of the rational topological complexity of these  spaces.
Considering the formal case, we first note:
\begin{theorem} \label{th1.1}
Let $S$ be an elliptic pure formal space. We have
$$ \TC_0(S)= 2\cat_0(S) +\chi_{\pi}(S)$$
where $\chi_{\pi}(S)= \dim \pi _{even}(S)\otimes \mathbb{Q}- \dim \pi _{odd} (S)\otimes \mathbb{Q}$.
\end{theorem}
Note that the ellipticity of S implies that $\chi_{\pi}(S) \leq 0$. If $\chi _{\pi}(S)=0$ then $S$ is formal and is called an $F_0$-space. In order to study the non-formal case we next establish  the following theorem:
\begin{theorem} \label{th1.2}
Let $S$ be a pure elliptic space and let $(\Lambda V,d)$ be its minimal Sullivan model. If there exists an extension $(\Lambda Z, d) \hookrightarrow (\Lambda  V,d) $ where $Z^{even}=V^{even}$ and $(\Lambda  Z,d)$ is the model of an $F_0$-space  $R$, then $$ \TC_0(S)\leq 2\cat_0(R) -\chi_{\pi}(S).$$	
If moreover there exists $k$ such that $d V \subset \Lambda ^k V$, then
$$ \TC_0(S)\leq 2\cat_0(S
) +\chi_{\pi}(S).$$	 \end{theorem}
When $S$ is elliptic and coformal (that is $d V \subset \Lambda ^2 V$) we have \cite{ref0}
\[\cat_0(S)= \dim \pi _ {odd}(S)\otimes \mathbb{Q}\]  and, under the hypothesis of the Theorem \ref{th1.2}, we obtain $$ \TC_0(S) \leq \dim \pi _{*}(S)\otimes \mathbb{Q}.$$
The second part of the article is dedicated to the study of the rational topological complexity of (pure) elliptic coformal spaces. In particular, the following theorem follows from Theorem \ref{th5.1} in Section \ref{sec}. The invariant $L_0$ which appears in the statement is a certain cuplength defined in the same section.
\begin{theorem} \label{th1.3}
	Let $S$ be a pure elliptic coformal space. Then $$\cat_0(S) + L_0(S) \leq \TC_0(S). $$
\end{theorem} 
This permits us, for instance, to see that the homogeneous space $S= \frac{SU(6)}{SU(3) \times SU(3)}$ satisfies $\TC_0(S)= \dim \pi _*(S) \otimes \mathbb{Q}=2\cat_0(S)+\chi_{\pi}(S)=5$.

Finally we establish the equality $\TC_0(S)= \dim \pi_* (S) \otimes \mathbb{Q}$ ($=2\cat_0(S)+\chi_{\pi}(S)$) for some special families of coformal spaces which include spaces for which Theorem \ref{th1.3} might not be sufficient to reach the equality.

\section{Preliminaries} \label{sec2}

Along this paper \textit{space} means a simply-connected CW-complex of finite type and we work over $\mathbb{Q}$. We refer to \cite{FHT} and \cite{FOT} as general references on Rational Homotopy Theory. A $cdga$ model of a space $S$ is a commutative (cochain) differential graded algebra $(A,d)$ with a chain  of quasi-isomorphisms $(A,d) \xleftarrow{\simeq} \cdots \xrightarrow{\simeq} A_{PL}(S)$ where $A_{PL}$ is the Sullivan functor of polynomial forms. Quasi-isomorphism here means a cdga morphism which induces an isomorphism in cohomology. In particular $H^*(A,d)= H^*(S;\mathbb{Q})$. If $(H^*(S;\mathbb{Q}),0)$ is a cdga model of $S$ then $S$ is said formal.

A Sullivan model of $S$ is a cdga model of the form $(\Lambda V,d)$ where $\Lambda V$ is free as a commutative graded algebra and there exists a decomposition  $V=\bigoplus _{k\geq 0} V(k)$ such that $d V(0)=0$ and $d(V(i) )\subset \bigoplus _{k\leq i-1} V(k)$. If moreover $d V\subset \Lambda ^{\geq 2} V$, the model is said minimal. In this case the graded vector space $V$ is isomorphic to the dual of the rational homotopy vector space $\pi_{*}(S)\otimes \mathbb{Q}$ of $S$.

A KS-extension of $(\Lambda V,d)$ is a cdga $(\Lambda V \otimes \Lambda Z,d)$ together with a cdga inclusion    $(\Lambda V,d) \hookrightarrow (\Lambda V \otimes \Lambda Z,d)$
where $Z= \bigoplus _{k\geq 0} Z(k)$ and $d (Z(k)) \subset \Lambda V \otimes ( Z(0) \oplus \cdots \oplus Z(k-1))$.
Such a KS-extension models a fibration $p :T \rightarrow S$ over S such that the quotient algebra $$ \left(\frac{\Lambda V \otimes \Lambda Z}{\Lambda ^+ V \otimes \Lambda Z}, \bar{d}\right) \cong (\Lambda Z,\bar{d})$$
is a model for the fiber of $p$.

Any cdga morphism $(\Lambda V,d) \rightarrow (A,d)$ admits a relative Sullivan model, that is a decomposition of the form \[
\xymatrix{
	(\Lambda V , d ) \ar@{_(->} _i[dr] \ar[rr] && ( A,d)\\
	& (\Lambda V \otimes \Lambda Z,D) \ar[ru] ^{ \xi} _- { \simeq}
} 
\] 
where $i$ is  a KS-extension and $\xi$ is a quasi-isomorphism.

Finally we recall the following lifting lemma:
		\begin{lemma}[Lifting lemma]
		Let consider the solid commutative diagram 
		\begin{equation*}
			\xymatrix{
				\Lambda Z \ar[r] \ar@{_(->}[d]_i & \Lambda V \ar@{>>}[d] ^{\simeq}_-\psi{}\\
				\Lambda Z  \otimes \Lambda X \ar@{-->}[ur]^{\exists r} \ar[r] & \Lambda W 
			}
		\end{equation*}
	in which $i$ is a KS-extension and $\psi$ is a surjective quasi-isomorphism. Then there is a (cdga) morphism $r: \Lambda Z\otimes \Lambda X\rightarrow \Lambda V$ which makes the two triangles commutative.
	\end{lemma}
	\section{Rational topological complexity of certain extensions} \label{sec3}
	Let $S$ be a simply-connected CW-complex of finite type. Recall that  the rational LS-category and topological complexity of $S$ are defined by
	\begin{center}
$\cat_0(S):=\cat(S_{\!{0}})$ and $\TC_0(S)=\TC(S_{\!{0}})$, 
	\end{center} 
where $S_{\!{0}}$ is the rationalization of $S$, and satisfy
\begin{center}
	$\cat_0(S)\leq\cat(S)$ and $\TC_0(S)\leq\TC(S)$. 
\end{center}

Suppose that $(\Lambda V , d)$ is a Sullivan model of $S$. Through the well-known Félix-Halperin characterization of $\cat_0$ from \cite{ref0} we have $\cat_0(S)=\cat(\Lambda V,d)$ where $\cat(\Lambda V,d)$ (or simply $\cat(\Lambda V)$) is the least integer $m$ such that the projection $$(\Lambda V,d) \rightarrow (\frac{\Lambda V}{\Lambda ^{>m} V},\bar{d})$$ admits a homotopy retraction. In \cite{ref5}, solving a conjecture posed by Jessup, Murillo and Parent \cite{JMP}, Carrasquel established a characterization of $\TC_0$ which is in the same spirit as Félix-Halperin's characterization of $\cat_0$. More precisely, let $\ker \mu _{\Lambda V}$ be the kernel of the multiplication $\mu _{\Lambda V}: \Lambda V \otimes \Lambda V \rightarrow \Lambda V $ of the algebra $(\Lambda V,d)$ and denote by $\rho _m$ the canonical projection $$ (\Lambda V \otimes \Lambda V  ,d)\rightarrow \left(\frac{\Lambda V \otimes \Lambda V}{(\ker \mu _{\Lambda V})^{m+1}}, \bar{d}\right).$$
	By denoting by $\TC(\Lambda V,d)$, or simply $\TC(\Lambda V)$, the least integer $m$ for which $\rho_m$ admits a homotopy retraction we have:
	\begin{theorem}{\cite[Theorem 8]{ref5}} $\TC(\Lambda V)=\TC_0(S).$
		\end{theorem}
	We recall that $\rho _ m$ admits a homotopy retraction if there exists a cdga morphism $r: \Lambda V\otimes \Lambda V \otimes \Lambda Z\rightarrow \Lambda V\otimes \Lambda V$ such that $r\circ i =id_{\Lambda V \otimes \Lambda V}$, where $ \Lambda V \otimes \Lambda V \stackrel{i}{\hookrightarrow } \Lambda V \otimes \Lambda V \otimes \Lambda Z$ is a relative Sullivan model of $\rho_m$. This is expressed by the following diagram 
	\begin{equation*}
		\xymatrix@=4pc{
			& \Lambda V \otimes \Lambda V \otimes \Lambda Z \ar@/_2.5pc/@{-->}[dl]_-r  \ar[d]^{\xi}_-{\simeq}\\
			\Lambda V \otimes \Lambda V \ar@{^{(}->}[ur]_-{i} \ar[r]_-{\rho _ m} & \frac{\Lambda V \otimes \Lambda V}{ (\ker \mu _{\Lambda V})^{m+1}}.
		}
	\end{equation*}
It follows also from Carrasquel's results that the invariants $\MTC(S)$ and $\HTC(S)$ defined in \cite{FLGPKT} and \cite{cjkl} can be  respectively identified to:
	\begin{itemize} 
			\item[•] $\MTC(\Lambda V)$ which is defined as the least integer $m$ for which $\rho _ m$ admit a homotopy retraction as a morphism of $\Lambda V \otimes \Lambda V$-module,
			\item[•] $\HTC(\Lambda V)$ which is defined as the least integer $m$ for which $\rho _ m$ is injective in cohomology.
		\end{itemize}
	 These invariants together with the rational zero-divisor cuplength $\zcl_{\mathbb{Q}}(H^*(\Lambda V )):= \zcl_{\mathbb{Q}}(H^*(S; \mathbb{Q}))$, defined as the largest integer $m$ such that there exist $\alpha _1, \cdots, \alpha _m \in \ker (H^*(\mu_{\Lambda V })) $ satisfying $\alpha _1 \cdots\alpha _m \neq 0 $, are ordered as follows: 
	$$\zcl_{\mathbb{Q}} (H^*(\Lambda V)) \leq \HTC (\Lambda V) \leq \MTC (\Lambda V) \leq \TC(\Lambda V).$$
When $S$ is formal, it follows from \cite{LM} that all the invariants coincide but as shown in \cite{FLGPKT} and \cite{cjkl} the invariants $\HTC$ and $\MTC$ can in general be larger than the zero-divisor cuplength. We refer to the survey \cite{carrasquelsurvey} for more details on the rational approximations of the topological complexity.
 
We now state and prove our main theorem in this section.
\begin{theorem} 
Let $(\Lambda V,d)$ be a Sullivan model and let $(\Lambda V \otimes \Lambda u, d)$ be an extension of $(\Lambda V,d)$ with $u$ is an element of odd degree. Then
\begin{itemize}
\item[(i)]$\TC(\Lambda V \otimes \Lambda u)\leq \TC(\Lambda V) +1 $. 
\item[(ii)]$\MTC(\Lambda V \otimes \Lambda u)\leq \MTC(\Lambda V) +1$.
\item[(iii)]$\HTC(\Lambda V \otimes \Lambda u)\leq \HTC(\Lambda V) +1$.
\end{itemize}
	\end{theorem}
	\begin{proof}
		Suppose that $\TC(\Lambda V)=m$. Then $m$ is the least integer for which the projection $\rho _m: \Lambda V \otimes \Lambda V \rightarrow\frac{\Lambda V \otimes \Lambda V}{ (\ker \mu _{\Lambda V})^{m+1}}$ admits a homotopy retraction $r$ as represented in the following diagram  
		\begin{equation*}
			\xymatrix@=4pc{
				& \Lambda V \otimes \Lambda V \otimes \Lambda W \ar@/_2.5pc/@{-->}[dl]_-r  \ar[d]^{\xi}_-{\simeq}\\
				\Lambda V \otimes \Lambda V \ar@{^{(}->}[ur]_-{i} \ar[r]_-{\rho _ m} & \frac{\Lambda V \otimes \Lambda V}{ (\ker \mu _{\Lambda V})^{m+1}}
			}
		\end{equation*}
		where $\xi$ is a quasi-isomorphism and $i$ is a relative model of $\rho_{m}$. Through the extension $\Lambda V \otimes \Lambda V \hookrightarrow \Lambda V \otimes \Lambda u \otimes \Lambda V \otimes \Lambda u \cong \Lambda V \otimes \Lambda V \otimes \Lambda u \otimes \Lambda u $ we obtain the morphisms  $\rho_{m} \otimes id_{\Lambda u} \otimes id_{\Lambda u}$, $i \otimes id_{\Lambda u} \otimes id_{\Lambda u}$, $ \xi \otimes id_{\Lambda u} \otimes id_{\Lambda u}$ and  $r \otimes id_{\Lambda u} \otimes id_{\Lambda u}$  which satisfy $ (\xi \otimes id_{\Lambda u} \otimes id_{\Lambda u}) ( i \otimes id_{\Lambda u} \otimes id_{\Lambda u})= \rho_m \otimes id_{\Lambda u} \otimes id_{\Lambda u}$ and $ (r \otimes id_{\Lambda u} \otimes id_{\Lambda u})( i \otimes id_{\Lambda u} \otimes id_{\Lambda u})=id$.
		Note also that $ \xi \otimes id_{\Lambda u} \otimes id_{\Lambda u}$ is a quasi-isomorphism. On the other hand, let consider $$(\Lambda V\otimes \Lambda u)\otimes ( \Lambda V\otimes \Lambda u) \hookrightarrow   (\Lambda V\otimes \Lambda u)\otimes ( \Lambda V\otimes \Lambda u) \otimes \Lambda Z \xrightarrow{\simeq} \frac{ (\Lambda V\otimes \Lambda u)\otimes ( \Lambda V\otimes \Lambda u)}{ (\ker \mu _{\Lambda V\otimes \Lambda u})^{m+2}}$$  a relative Sullivan model of $\rho _{m+1}: (\Lambda V\otimes \Lambda u)\otimes ( \Lambda V\otimes \Lambda u) \rightarrow \frac{(\Lambda V\otimes \Lambda u)\otimes ( \Lambda V\otimes \Lambda u)}{(\ker \mu _{\Lambda V\otimes \Lambda u})^{m+2}}$.\\
		Since $|u|$ is odd, we have $(\ker \mu _{\Lambda u})^2=0$. This fact together with  $$\ker \mu _{\Lambda V \otimes \Lambda u}\cong \ker \mu _{\Lambda V} \otimes \Lambda u \otimes \Lambda u + \Lambda V \otimes \Lambda V \otimes \ker \mu _{\Lambda u},$$
		implies that the composition 
		$$(\ker \mu _{\Lambda V\otimes \Lambda u})^{m+2} \hookrightarrow (\Lambda V\otimes \Lambda V)\otimes ( \Lambda u\otimes \Lambda u) \xrightarrow {\rho _m \otimes id_{\Lambda u} \otimes id_{\Lambda u}} \frac{\Lambda V\otimes \Lambda V}{(\ker \mu _{\Lambda V})^{m+1}}\otimes \Lambda u \otimes \Lambda u$$
		is trivial. Therefore $\rho_{m} \otimes id_{\Lambda u} \otimes id_{\Lambda u}$ can be factorized as 
		\[
		\xymatrix{
			(\Lambda V \otimes \Lambda V )\otimes (\Lambda u \otimes \Lambda u)\ar[dr]_-{\rho_{m+1}} \ar[rr]^{\rho _m \otimes id_{\Lambda u} \otimes id_{\Lambda u}} && \frac{\Lambda V \otimes \Lambda V}{(\ker \mu _{\Lambda V})^{m+1}} \otimes (\Lambda u \otimes \Lambda u)\\
			&\frac{(\Lambda V \otimes \Lambda u) \otimes( \Lambda V \otimes \Lambda u)}{ (\ker \mu _{\Lambda V \otimes \Lambda u})^{m+2}}. \ar[ru] 
		}
		\]
	   Using the relative model of $\rho_{m+1}$ we obtain the diagram
		\[
		\xymatrix{
			\Lambda V \otimes \Lambda V  \otimes \Lambda u \otimes \Lambda u \ar@{^(->}[d] \ar[drr]^{ \rho _ m \otimes id_{\Lambda u} \otimes id_{\Lambda u}}\\
			(\Lambda V \otimes \Lambda u)\otimes ( \Lambda V \otimes \Lambda u)\otimes \Lambda Z \ar[r] & \frac{(\Lambda V \otimes \Lambda u)\otimes ( \Lambda V \otimes \Lambda u)}{(\ker \mu _{ \Lambda V \otimes \Lambda u})^{m+2}} \ar[r] & \frac{\Lambda V \otimes \Lambda V}{ (\ker \mu _{\Lambda V})^{m+1}} \otimes \Lambda u \otimes \Lambda u. 
		}
		\]
		
		We can then form the following commutative solid diagram
		\begin{equation*}
			\xymatrixcolsep{0.02pc}\xymatrix{
				 \Lambda V \otimes \Lambda V \otimes \Lambda u\otimes \Lambda u
				\ar@{^{(}->}[d] \ar[rr]^{i \otimes id_{\Lambda u} \otimes id _{\Lambda u}}  &&(\Lambda V \otimes \Lambda V \otimes \Lambda W) \otimes (\Lambda u \otimes \Lambda u) \ar[d]^{\xi \otimes id_{\Lambda u} \otimes id_{\Lambda u}} _-{\simeq}
				\\
				(\Lambda V \otimes \Lambda u )\otimes (\Lambda V \otimes \Lambda u )\otimes \Lambda Z \ar@{-->}[urr]^{\tau}  \ar[r] & \frac{(\Lambda V \otimes \Lambda u )\otimes (\Lambda V \otimes \Lambda u )}{(\ker \mu _{\Lambda V \otimes \Lambda u})^{m+2}} \ar[r] & \frac{\Lambda V \otimes \Lambda V}{(\ker \mu _{\Lambda V})^{m+1}} \otimes (\Lambda u \otimes \Lambda u)}
		\end{equation*}
		and, by the lifting lemma, there is a morphism $\tau$ which makes the complete diagram commutative. As a result, the composition  $(r \otimes Id_{\Lambda u }\otimes Id_{\Lambda u } ) \circ \tau$ is the desired homotopy retraction of $\rho _{m+1}: (\Lambda V \otimes \Lambda u)\otimes (\Lambda V \otimes \Lambda u)\rightarrow \frac{(\Lambda V \otimes \Lambda u)\otimes (\Lambda V \otimes \Lambda u)}{(\ker \mu _{\Lambda V \otimes \Lambda u})^{m+2}}$, and we conclude that $\TC(\Lambda V \otimes \Lambda u)\leq m+1$.\\
		Assertions (ii) and (iii) are obtained in the same way as (i) by taking respectively a $\Lambda V \otimes \Lambda V$-retraction and a linear differential retraction instead of a (cdga) homotopy retraction. 
	\end{proof}
	By induction we can generalize this result to get the following corollary, which can be seen as a version for $\TC$ of \cite[Prop 30.7]{FHT}. 
	\begin{corollary} \label{corollary3.1}
		If $(\Lambda V,d) \hookrightarrow (\Lambda V \otimes \Lambda U, d) $ is an extension of $(\Lambda V,d)$, where $U$ is a graded vector space concentrated in odd degrees with $d(U) \subset \Lambda V$ and $\dim U=n$, then
		\begin{itemize}
			\item[(i)] $\TC(\Lambda V \otimes \Lambda U) \leq \TC(\Lambda V ) +n$.
			\item[(ii)] $\MTC(\Lambda V \otimes \Lambda U) \leq \MTC(\Lambda V ) +n.$
			\item[(iii)] $\HTC(\Lambda V \otimes \Lambda U) \leq \HTC(\Lambda V ) +n$.  
		\end{itemize}
	\end{corollary}
\section{An upper bound of $\TC_0$ for certain pure elliptic spaces}
We recall that a space $S$ is said elliptic if $\pi_*(S) \otimes \mathbb{Q}$ and $H^*(S; \mathbb{Q})$ are finite dimensional and pure if it admits a Sullivan model $(\Lambda V,d)$ such that $d V^{even}=0$ and $dV^{odd} \subset \Lambda V^{even}$. If $(\Lambda V,d)$ is minimal then $V\cong \pi_*(S) \otimes \mathbb{Q}$ and the homotopy characteristic of $S$ given by $\chi_{\pi}(S)= \dim \pi_{even}(S) \otimes \mathbb{Q}-\dim \pi_{odd}(S) \otimes \mathbb{Q} $ coincides with $\chi_{\pi}(\Lambda V)= \dim V^{even}-\dim V^{odd} $.
Abusing language we use the terminology "pure", "formal", "elliptic" for both space and its minimal Sullivan model. 
Recall that by ellipticity we always have $\chi_{\pi}(\Lambda V) \leq 0$ and that an elliptic minimal model $(\Lambda V,d)$ is called an $F_0$-model if $\chi_{\pi}(\Lambda V)=0$. Theorem \ref{th1.1} follows from  

	\begin{theorem} If $(\Lambda V,d)$ is a pure elliptic minimal Sullivan model which is formal then 
		\begin{equation*} \label{th4.1}
			\TC(\Lambda V)= 2 \cat(\Lambda V)+ \chi_{\pi}(\Lambda V).
		\end{equation*}
		\end{theorem}
	\begin{proof}
		Since $(\Lambda V,d)$ is formal, it can be decomposed as 
		$$ (\Lambda V,d)= (\Lambda V',d) \otimes (\Lambda (z_1, \cdots, z_l), 0)$$
		where $(\Lambda V',d)$ is an $F_0$-model and $z_i$ are all generators of odd degree (see for instance,  \cite[Lemma 1.5]{AM}). As $( \Lambda V,d)$ is formal, we have by \cite{LM}
		\begin{eqnarray*}
			\TC( \Lambda V)&=& \MTC( \Lambda V)=\zcl _{\mathbb{Q}}(H^*(\Lambda V)).
		\end{eqnarray*} 
		
		As is well-known we also have $\cat (\Lambda V)=\cl_{\mathbb{Q}}(H^*(\Lambda V))$ \cite{FHT}.
		By \cite{JMP} (see also \cite[th.12]{CJC}) we have that $ \MTC( \Lambda V)=\MTC( \Lambda V')+l$. Moreover since $( \Lambda V',d)$ is an $F_0$-model, we have $H^*(\Lambda V')=H^{even}(\Lambda V')$. Therefore we have $\zcl_{\mathbb{Q}}(H^*(\Lambda V'))=2\cl_{\mathbb{Q}}( H^*(\Lambda V')) $ \cite[Corollary 30]{CJV}. Since $\Lambda V'$ is formal, it follows that $\MTC( \Lambda V')=2 \cdot \cat(\Lambda V')$ and therefore  
		$$\TC( \Lambda V)=  2\cat( \Lambda V')+l .$$
		On the other hand, by the additivity of $\cat_0$ with respect to the product (see for instance, \cite[Th 30.2]{FHT}), we have  
		\begin{eqnarray*}
			2\cat(\Lambda V)+ \chi_{\pi}( \Lambda V)&=& 2 (\cat(\Lambda V') +l) -l\\
			&=& 2 \cat(\Lambda V') +l.
		\end{eqnarray*} 
		Finally we get $\TC( \Lambda V)=2\cat(\Lambda V) +\chi_{\pi}( \Lambda V)$.
	\end{proof}
Our next result, which corresponds to Theorem \ref{th1.2} from the introduction, will permits us to obtain, under some conditions, an upper bound for $\TC(\Lambda V)$ in the non-formal case. 
\begin{theorem}\label{th4.2}
	Let $(\Lambda V,d)$ be a pure elliptic minimal Sullivan model. If there exists an extension $(\Lambda Z, d) \hookrightarrow (\Lambda  V,d) $ where $Z^{even}=V^{even}$ and $(\Lambda  Z,d)$ is an $F_0$-model, then $$ \TC(\Lambda V)\leq 2\cat(\Lambda Z) -\chi_{\pi}(\Lambda V).$$	
	If moreover there exists $k$ such that $d V \subset \Lambda ^k V$, then
$$ \TC(\Lambda V)\leq 2\cat(\Lambda V) +\chi_{\pi}(\Lambda V).$$	
\end{theorem}

\begin{remark}
{\rm	According to \cite[Lemma 3.3]{J} (see the remark after the proof of Lemma 3.3), if $V^{even}$ is concentrated in a single degree then it is always possible to construct an $F_0$-model $(\Lambda Z,d)$ as required in hypothesis of Theorem \ref{th4.2}}
\end{remark}
\begin{proof}[Proof of Theorem \ref{th4.2}]
	In the conditions of the theorem we can suppose that $V= Z\oplus U$ where $U$ is a subspace of $V$ concentrated in odd degrees. By Corollary \ref{corollary3.1}, we have  \begin{eqnarray*}
		\TC ( \Lambda V)&=& \TC( \Lambda (Z\oplus U),d)\\
		&\leq & \TC( \Lambda Z,d) + \dim U.
	\end{eqnarray*}
Since $(\Lambda Z,d)$ is an $F_0$-model, we have, by Theorem \ref{th4.1}, $\TC( \Lambda Z)=2\cdot \cat(\Lambda Z)$ and we obtain
$$\TC ( \Lambda V)\leq  2 \cdot \cat(\Lambda Z) + \dim U. \quad (*)
$$

On the other hand $\dim V^{even}=\dim Z^{even}=\dim Z^{odd}$ implies that
 \begin{eqnarray*}
	\dim U&=& \dim V -\dim Z\\
	&=& \dim V^{even} +\dim V^{odd} - \dim Z^{even} - \dim Z^{odd} \\
	&= & \dim V^{odd}- \dim V^{even}\\
	&=& - \chi_{\pi}(\Lambda V). 
\end{eqnarray*}
We then obtain $\TC (\Lambda V)\leq 2 \cat(\Lambda Z) - \chi_{\pi}(\Lambda V)$.

We now suppose that the differential $d$  is homogeneous of rank $k$  $(d=d_k)$. By  Lechuga-Murillo formula for the rational LS-category of such an elliptic  space \cite{LM02}, if $\dim V^{even}=n$  and $\dim V^{odd}=m$, thus $$\cat(\Lambda Z)= n(k-2)+n=n(k-1) \quad \text{and} \quad \cat(\Lambda V)= n(k-2)+m.$$
As we have
\begin{eqnarray*}
2 \cdot \cat(\Lambda V) + \chi _{\pi} ( \Lambda V) &=& 2[ n(k-2) +m]+n-m\\
&=& 2n(k-1) -2n +2m +n-m\\
&=& 2n(k-1) +m-n\\
&=& 2 \cat(\Lambda Z) +\dim U,
\end{eqnarray*}
inequality $(*)$ finally implies $\TC(\Lambda V)\leq 2\cat(\Lambda V)+ \chi_{\pi}( \Lambda V)$. 
\end{proof}

	\section{The coformal case}
	Along this section, we consider  a pure coformal model $(\Lambda V,d)$ where $\dim V$ is finite. Recall that coformal means that $d V\subset \Lambda ^2 V$. We denote by  $X=V^{even}$ and $Y=V^{odd}$. As the model is pure we have $dX=0$ and $dY\subset \Lambda X$. We note that $(\Lambda V,d)$ is not required to be elliptic. However, given a basis ${\mathcal B}= \{x_1, \cdots , x_n\}$ of $X$, we associate with $(\Lambda V,d)$ the elliptic extension $$ \Lambda W_{\mathcal B} =\Lambda (X \oplus Y \oplus U),$$
	where $U$ is the graded vector space generated by $u_1,\cdots ,u_n$ with $du_i=x_i^2$ for all $i=1,\cdots , n$. 
Let point out that the elliptic extension $(\Lambda W_{\mathcal B} ,d)$ associated to $(\Lambda V,d)$ satisfies the inequality in the proposition below:
\begin{proposition} \label{pr5.1}
The elliptic extension $(\Lambda W_{\mathcal B} ,d)$ of the model $(\Lambda V,d)$ satisfies 
$$\TC( \Lambda W _{\mathcal B}) \leq \dim W_{\mathcal B}=2\dim V^{even} + \dim V^{odd}.$$
\end{proposition}
\begin{proof}
It suffices to consider the extension $(\Lambda(X\oplus U),d)\hookrightarrow (\Lambda W_{\mathcal B},d)$ and then the result follows by Theorem \ref{th4.2}. 
\end{proof}
	We will show later that this construction permits us to obtain a lower bound of the rational topological complexity of $(\Lambda V,d)$. In our calculations, it will be useful to reduce the algebra $\Lambda W _{\mathcal B}$ to a smaller algebra $A_{\mathcal B}$ through the surjective quasi-isomorphism 
	$$ \varphi: (\Lambda W _ {\mathcal B},d) \twoheadrightarrow \left( A_ {\mathcal B}=\frac{\Lambda (x_i)}{(x_i^2)} \otimes \Lambda Y,\bar{d} \right)$$
	defined by $\varphi(x)=x$, $ \varphi(y)=y$, $ \varphi(u)=0$, for $x\in X$, $y\in Y$ and $u\in U$.
	In the sequel we will simply write $\Lambda W$ (resp. $A$) instead of $\Lambda W_{\mathcal B}$ (resp. $A_ {\mathcal B}$) if no confusion arises.
	\subsection{The cocycle $\Omega$} \label{class}
	Recall that an elliptic space is a Poincaré duality space.
	In \cite{L.M}, Lechuga and Murillo gave an efficient method  to obtain an explicit expression of a cocycle representing the fundamental class of any pure elliptic model.
	We first use this process to construct a representing cocycle $\omega$ for the fundamental class of the elliptic extension $\Lambda W$ of $\Lambda V$ described above.
	Let $\bar{X}=sX$ be the suspension of $X$ with $d$ extended by $d \bar{x}=x$, $\forall x \in X$. Using the same notations as above we have  $W^{even}= X$ and $W ^{odd}=Y\oplus U$.
	
     We work in $\Lambda W \otimes \Lambda \bar{X}$. Let $\{y_1, \cdots , y_m\}$ be a basis of $Y$. Since $d$ is a quadratic differential, for each $j=1,\cdots, m$, we may express $d(y_j)$ as
	\begin{eqnarray*}
		d(y_j)&=& \sum _k  \alpha ^j _k x^2_k+ \sum _{p<q} \beta ^j _{p,q}x_{p} \cdot x_{q}\\
		&=& d\{ \sum _k  \alpha ^j _k x_k \cdot \bar{x}_k + \sum _{p<q} \beta ^j _{p,q}x_{p} \cdot \bar{x}_{q}\}
	\end{eqnarray*} 
	where $\alpha^j_k ,\beta^j_{p,q} \in \mathbb{Q}$. In addition \begin{eqnarray*}
		d(u_i)&=& x_i^2\\
		&=& d(x_i \cdot \bar{ x}_i).
	\end{eqnarray*}
	According to \cite{L.M}, $\omega$ can be obtained as the cofficient of $\prod _{i=1}^n \bar{x}_i$ after developing the product
	$$ \prod ^m _{j=1} (y_j -\sum _k  \alpha ^j _k x_k \cdot \bar{x}_k - \sum _{p<q} \beta ^j _{p,q}x_{p} \cdot \bar{x}_{q}) \cdot \prod _{i=1}^n(u_i - x_i \cdot \bar{x}_i).$$
	It is clear to see that
	$$\varphi( \omega)= (-1)^n x_1 \cdots x_n \cdot y_1 \cdots y_m.$$
	By setting $x_{[n]}=x_1 \cdots x_n $, $y_{[m]}= y_1 \cdots y_m$ and $\omega _A=(-1)^n \cdot \varphi( \omega)=x_{[n]} \cdot y_{[m]}$, we remark that $\omega _A$ denote the unique (up to a scalar) representing element for the fundamental class of the algebra $A$.
	
	By adapting the technique above, we shall construct an important cocycle  $\Omega \in (\ker \mu _{\Lambda W})^{m+n}$, which will be a key ingredient to obtain a lower bound of the $\TC(\Lambda V)$. More precisely, we consider a second copy $\Lambda W'$ of $\Lambda W$ and we consider $\Lambda W \otimes \Lambda W' \cong \Lambda(W \oplus W')$ writing $x'_i , u'_i$  and $ y'_j$ for the elements corresponding to $x_i , u_i$ and $y_j$. Working in $\Lambda W \otimes \bar{X} \otimes \Lambda W' \otimes \bar{X}'$, the odd generators of $\ker \mu_{\Lambda W }$, where $\mu_{\Lambda W }$ is the multiplication over $\Lambda W $, satisfy
	\begin{eqnarray*}
		d(u_i-u_i')&=& d[(x_i-x'_i)( \bar{x}_i+\bar{x}_i')],  \text{     for all  } i=1,\cdots ,n,
	\end{eqnarray*} 
	and
	\begin{eqnarray*}
		d(y_j-y_j')&=& \sum _k \alpha ^j _k(x^2_{k}- {x '}^2_{k}) + \sum _{p< q} \beta ^j _{p,q}(x_{p} x_{q} -x'_{p} x'_{q}) \\
		&=&d[ \sum _k \alpha ^j _k (x_{k}- {x '}_{k})(\bar{ x }_{k} + \bar{ x}'_{k} ) +\frac{1}{2} \sum _{p< q}  \beta ^j _{p,q}(x_{p}- {x '}_{p})(\bar{ x }_{q} + \bar{ x}'_{q} )\\
		&+& \frac{1}{2} \sum _{p< q}  \beta ^j _{p,q}(x_{q}- {x '}_{q})(\bar{ x }_{p} + \bar{ x}'_{p} )], \quad \text{for all} \quad j=1,\cdots ,m.
	\end{eqnarray*}
	By setting
	$$\phi _j=  \sum _k \alpha ^j _k (x_{k}- {x '}_{k})(\bar{ x }_{k} + \bar{ x}'_{k} ) +\frac{1}{2} \sum _{p< q}  \beta ^j _{p,q}(x_{p}- {x '}_{p})(\bar{ x }_{q} + \bar{ x}'_{q} )+\frac{1}{2} \sum _{p< q}  \beta ^j _{p,q}(x_{q}- {x '}_{q})(\bar{ x }_{p} + \bar{ x}'_{p} )$$
	we obtain that, the elements
	\begin{itemize}
		\item[•] $y_j -y'_j- \phi _j$, with $j=1, \cdots , m$
		\item[•] $u_i -u'_i- (x_i-x'_i)(\bar{x_i} + \bar{x'_i})$, with $i=1,\cdots , n$
	\end{itemize}
	are all cocycles in $\Lambda W   \otimes \Lambda \bar{X}\otimes \Lambda W '  \otimes \Lambda \bar{X}'$.\\ 
	We define $\Omega$ to be the coefficient of $ \prod _{i=1}^n( \bar{x}_i+\bar{x}_i')$ gotten after developing \\  
	$$\prod _{j=1}^m [y_j-y'_j -\phi _j] \cdot \prod _{i=1}^n[u_i-u'_i-(x_i-x'_i)( \bar{x}_i+\bar{x}'_i)].$$
	Remark that $$\Omega _A:=(-1)^n \cdot (\varphi \otimes \varphi)(\Omega)= \prod _{i=1}^n (x_i-x'_i)\cdot \prod_{j=1}^m (y_j-y_j').$$
	We will see through the following lemma, that $\Omega$ is a cocycle with non-zero cohomology class.
	\begin{lemma}
		$\Omega $ is a cocycle in $(\ker \mu _{\Lambda W })^{m+n}$ with non-zero cohomolgy class.
	\end{lemma}
	\begin{proof}
		By construction it is clear that  $\Omega \in (\ker \mu _{\Lambda W })^{n+m}$. Now, considering the decomposition 
		$$\Lambda W  \otimes \Lambda W'  \otimes \Lambda^{\leq n} ( \bar{X}\oplus \bar{X}')= \Lambda W  \otimes \Lambda W'  \otimes \Lambda ^n (\bar{X}\oplus \bar{X}') \oplus \Lambda W  \otimes \Lambda W'  \otimes \Lambda ^{<n} (\bar{X}\oplus \bar{X}'),$$ 
		we have
		\begin{equation*}
			\prod _{j=1}^m [y_j-y'_j -\phi _j] \cdot \prod _{i=1}^n [u_i-u'_i -(x _i -x' _i) (\bar{x} _i +\bar{x}' _i)] = \Omega  \cdot \prod _{i=1}^n(\bar{x} _i +\bar{x}' _i) + \sum _{0\leq i< n} \Theta _i\cdot \prod _{k=1}^i(\bar{x}_{i_k} + \bar{x } '_{i_k}),
		\end{equation*}
		where  $\Theta _i \in \Lambda W \otimes \Lambda W'$, for $i=1, \cdots , n-1$.\\ 
		Applying the differential $d$ to the equation above we get, after identification, $d \Omega =0$.\\
		To see that $[\Omega] \neq 0$, we proceed by contradiction and suppose that $[\Omega]=0$. Then $[\Omega _A]=0$. Let  $A'$ be a second copy of the algebra $A$. Recall that $\omega' _A=x'_{[n]}\cdot y'_{[m]}$ with $x'_{[n]}=x_1' \cdots x_n'$ and $y'_{[m]}=y'_1\cdots y'_m$  is the unique element (up to a scalar)  representing the fundamental class of $A'$. Since $[\Omega _A]=0$,
		we have \begin{equation} \label{equation2}
			[\Omega _A]\cdot [\omega' _A]=0.
		\end{equation}    
		On the other hand 
		\begin{eqnarray*}
			[\Omega _A]\cdot [\omega' _A]&=& [ \prod _{i=1}^n(x_i-x_i')\cdot \prod _{j=1}^m(y_j-y_j')]\cdot [x'_{[n]}\cdot y'_{[m]}]\\
			&=& [\prod _{i=1}^n(x_i-x_i')\cdot x'_{[n]}\cdot \prod _{j=1}^m(y_j-y_j') \cdot y'_{[m]}]\\
			&=&  [x_{[n]} x'_{[n]} y_{[m]} y'_{[m]}]\\
			&=& [\omega_A][\omega' _A]\\
			&\neq &
			 0,
		\end{eqnarray*}
		which contradicts (\ref{equation2}).
	\end{proof}
	\subsection{Lower bound of $\TC( \Lambda V	)$} \label{sec}
	As previously we consider a pure coformal model $(\Lambda V,d)$ with $\dim V<\infty$ and the elliptic extension
	$$ \Lambda W_ {\mathcal B } = \Lambda (X \oplus Y \oplus U),$$
	associated to the basis ${\mathcal B }= \{ x_1 , \cdots , x_n\}$ of $X=V^{even}$.
	We recall that $U$ is the graded vector space generated by $u_1, \cdots , u_n$, with $d u_i= x_i^2$, and we also fix a basis $\{ y_1,\cdots, y_m \}$ of $Y=V^{odd}$. Since $d$ is a quadratic differential, it splits into as the sum of the following operators 
	$$d_{p,q}:\Lambda ^p X \otimes \Lambda ^q(Y\oplus U) \rightarrow  \Lambda ^{p+2} X \otimes \Lambda ^{q-1}(Y\oplus U).
	$$
	We then obtain a (lower) bigradation on the cohomology of $\Lambda W$ given by  
	$$H_{p,q}(\Lambda W)= \frac{\ker (d_{p,q})}{Im (d_{p-2,q+1})}.$$
	Recall the quasi-isomorphism $$\varphi: \Lambda W _ {\mathcal B}= \Lambda (x_i, y_j, u_i) \rightarrow A= \frac{\Lambda (x_i)}{(x_i^2)} \otimes \Lambda (y_j).$$
	Denoting by $M_p(X)$ the vector space $\varphi (\Lambda^p X)$, the differential of $A$ splits into the sum of
	$$d_{p,q}: M_p(X) \otimes \Lambda ^q Y \longrightarrow M_{p+2}(X) \otimes \Lambda ^{q-1} Y
	$$
	and we obtain a lower bigradation of $H(A)$ given by
	$$H_{p,q}(A)= \frac{\ker (d_{p,q})}{Im (d_{p-2,q+1})}.$$
	This yields an isomorphism compatible with the bigradation structure
	$$H_{*,*}(\varphi): H_{*, *}(\Lambda W _ {\mathcal B }) \xrightarrow{\cong } H_{*,*}(A).$$
We then define the following cuplength which takes into account only the cohomology classes whose wordlength in $X=V^{even }$ is odd.
	$$
	L(\Lambda V,{\mathcal B }):= max\{ r: \exists \alpha _1 , \cdots , \alpha _r \in H_{odd, *}(A) \text{ such that } \alpha _1 \cdots \alpha _r \neq 0 \text{ in } H(A)\}.$$
	This gives a lower bound to the rational topological complexity through the following theorem.
	\begin{theorem} \label{th5.1}
		Let $(\Lambda V,d)$ be a pure coformal model with $\dim V <\infty$. For every basis ${\mathcal B}$ of $V^{even}$, we have
		$$dim (V^{odd})
		+ L(\Lambda V, {\mathcal B }) \leq \TC(\Lambda V).$$
	\end{theorem} 
 If $(\Lambda V,d)$  is the minimal Sullivan model of $S$ then the invariant $L_0(S)$ mentioned in Theorem \ref{th1.3} is defined as $$L_0(S)=\sup\{ L(\Lambda V, {\mathcal B }): {\mathcal B } \text{ is a basis for } V^{even}\}$$
 and Theorem \ref{th1.3} of the introduction follows directly from Theorem \ref{th5.1} together with the fact that, for an elliptic coformal space $\cat_0(S)=\dim \pi_{odd}(S)\otimes \mathbb{Q}$. 
 
Before to begin proving our result, we need to set some lemmas, which will be the keys of the demonstration.\\
As before we consider $\Lambda W'$ a second copy of $\Lambda W$ as well as a second copy $A'$ of $A$.     
	\begin{lemma}\label{lemma5.2} Let $1\leq q\leq m$ and $1\leq j_1 < \cdots < j_q\leq m$. In both $\Lambda W \otimes \Lambda W'$ and $A\otimes A'$ we have
		$$ \prod _{j=1}^m(y_l -y'_l)  \cdot y_{j_1} \cdots y_{j_q}= \prod _{l=1}^m(y_l -y'_l) \cdot  y'_{j_1} \cdots y'_{j_q}.$$
	\end{lemma}
	\begin{proof}
		First, we have
		 
		\begin{eqnarray*}
			\prod _{l=1} ^m (y_l-y'_l) \cdot y_{j_1}
			& = &\prod_ {l=1} ^{j_1-1} (y_l-y'_l) \cdot(y_{j_1}-y'_{j_1}) \cdot \prod _{l={j_1+1}}^m (y_l-y'_l) \cdot y_{j_1}\\
			& =&\,(-1)^{\varepsilon} \cdot \prod_ {l=1} ^{j_1-1} (y_l-y'_l) \cdot(y_{j_1}-y'_{j_1}) \cdot y_{j_1} \cdot \prod _{l={j_1+1}}^m (y_l-y'_l).
		\end{eqnarray*}
	where $\varepsilon=|\prod _{l={j_1+1}}^m (y_l-y'_l) ||y_{j_1}|$.
		As $|y_{j_1}|$ is odd, $(y_{j_1}-y'_{j_1}) \cdot y_{j_1} =(y_{j_1}-y'_{j_1}) \cdot y'_{j_1}$ and it follows that 
		\begin{eqnarray*}
			\prod _{l=1} ^m (y_l-y'_l) \cdot y_{j_1}&=&(-1)^{\varepsilon} \cdot \prod_ {l=1} ^{j_1-1} (y_l-y'_l) \cdot (y_{j_1}-y'_{j_1}) \cdot y'_{j_1} \cdot \prod _{l={j_1+1}}^m (y_l-y'_l)\\
			&=&\prod_ {l=1} ^{j_1-1} (y_l-y'_l) \cdot (y_{j_1}-y'_{j_1}) \cdot \prod _{l={j_1+1}}^m (y_l-y'_l) \cdot y'_{j_1}\\
			&=&\prod _{l=1} ^m (y_l-y'_l) y'_{j_1}.
		\end{eqnarray*}
		Hence, the result is true for $y_{j_1}$.\\
		Next, by induction suppose that the equality is satisfied up to the rank $q-1$, which means
		$$\prod _{l=1}^m(y_l -y'_l)\cdot y_{j_1} \cdots y_{j_{q-1}}= \prod _{l=1}^m(y_l -y'_l)\cdot y'_{j_1} \cdots y'_{j_{q-1}}.$$ 
		Thus, we get
		\begin{eqnarray*}
			\prod _{l=1}^m(y_l -y'_l)\cdot y_{j_1} \cdots y_{j_{q}}&=& (\prod _{l=1}^m(y_l -y'_l) \cdot y_{j_1} \cdots y_{j_{q-1}}) \cdot y_{j_q}\\
			&=& ( \prod _{l=1}^m(y_l -y'_l) \cdot y'_{j_1} \cdots y'_{j_{q-1}}) \cdot y_{j_q}\\
			&=& (-1)^{ |\prod _{l=1}^m(y_l -y'_l)|| y'_{j_1} \cdots y'_{j_{q-1}}|  }\cdot y'_{j_1} \cdots y'_{j_{q-1}}  \cdot ( \prod _{l=1}^m(y_l -y'_l) \cdot y_{j_q}).
		\end{eqnarray*}
		By the same calculation as in the first step, we have $$\prod _{l=1}^m(y_l -y'_l) \cdot y_{j_q}= \prod _{l=1}^m(y_l -y'_l) \cdot y'_{j_q}.$$ Then
		\begin{eqnarray*}
			\prod _{l=1}^m(y_l -y'_l)\cdot y_{j_1} \cdots y_{j_{q}}&=& (-1)^{ |\prod _{l=1}^m(y_l -y'_l)|| y'_{j_1} \cdots y'_{j_{q-1}}|  }  \cdot y'_{j_1} \cdots y'_{j_{q-1}} \cdot ( \prod _{l=1}^m(y_l -y'_l) \cdot y'_{j_q})\\
			&=&\prod _{l=1}^m(y_l -y'_l)\cdot  y'_{j_1} \cdots y'_{j_{q}}.
		\end{eqnarray*}
	\end{proof} 
	\begin{lemma} \label{lemma02} Let $z_1, \cdots, z_r \in A$ such that for any $1\leq i\leq n ,$ $z_i \in M _{2s_i +1}(X) \otimes \Lambda Y$ with $s_i \in \mathbb{N}$. We have 
		$$ \Omega _A  \cdot (z_1-z'_1) \cdots (z_r-z'_r)=2^r \cdot \Omega _A \cdot z_1 \cdots z_r
		$$
		where $\Omega _A= (-1)^n \cdot (\varphi \otimes \varphi )(\Omega).$
	\end{lemma}
	\begin{proof}
		An element $z \in M_{2s+1}(X)\otimes \Lambda Y$ can be written as 
		$$\sum _{I,J} \alpha_{I,J} x_I \cdot y_J$$
		where $\alpha _{I,J} \in \mathbb{Q}$, $I=\{ 1\leq i_1< i_2< \cdots < i_{2p+1}\leq n \}$, $J=\{1 \leq j_1 < j_2 \cdots < j_q \leq m \}$, $x_I:= x_{i_1} \cdots x_{i_{2p+1}}$ and $y_J:= y_{j_1} \cdots y_{j_q}$. By setting $\hat{I}=\{1, \cdots ,n\} \setminus I$, and using the oddness of $card(I)$ we have
		\begin{eqnarray*}
			\Omega _A  \cdot z&=& \underset{I,J}{\sum} \alpha_{I,J} \prod _{k=1}^n(x_k-x'_k) \cdot \prod _{l=1}^m(y_l-y'_l) \cdot x_I \cdot y_J \\
			&=&-\underset{I,J}{\sum} \alpha_{I,J} \cdot x_I \cdot x'_I \prod _{k \in\hat{I}}(x_k-x'_k) \cdot \prod _{l=1}^m(y_l-y'_l) \cdot y_J. 
		\end{eqnarray*}
		Analogously, and using Lemma \ref{lemma5.2}, we get 
		\begin{eqnarray*}
			\Omega _A \cdot z'&=& \sum _{I,J} \alpha_{I,J} \prod _{k=1}^n(x_k-x'_k) \cdot \prod _{l=1}^m(y_l-y'_l) \cdot x'_I \cdot y'_J \\
			&=&  \sum _{I,J} \alpha_{I,J} \cdot x_I \cdot x'_I \prod _{k \in\hat{I}}(x_k-x'_k) \cdot \prod _{l=1}^m(y_l-y'_l) \cdot y'_J\\
			&=&  \sum _{I,J} \alpha_{I,J} \cdot x_I \cdot x'_I \prod _{k \in\hat{I}}(x_k-x'_k) \cdot \prod _{l=1}^m(y_l-y'_l) \cdot y_J. 
		\end{eqnarray*}
		Therefore 
		\begin{eqnarray} \label{eq3.4}
			\Omega _A \cdot (z-z')&=&-2\sum _{I,J} \alpha_{I,J} \cdot x_I \cdot x'_I \prod _{k \in\hat{I}}(x_k-x'_k) \cdot \prod _{l=1}^m(y_l-y'_l) \cdot y_J\nonumber\\
			&=& 2 \Omega _A \cdot z.
		\end{eqnarray}
		By induction, we suppose that
		$$\Omega _A  \cdot \prod _{k=1}^{r-1}(z_k-z'_k)= 2^{r-1}\Omega _A  \cdot \prod _{k=1}^{r-1}z_k$$
		and we have
		\begin{eqnarray*}
			\Omega _A  \cdot \prod _{k=1}^{r}(z_k-z'_k)&=&\Omega _A  \cdot \prod _{k=1}^{r-1}(z_k-z'_k)\cdot (z_r-z'_r) \\
			&=&2^{r-1}\Omega _A \cdot \prod _{k=1}^{r-1}z_k \cdot (z_r-z'_r)\\
			&=& 2^{r-1} (-1)^{|\Omega _A ||\prod _{k=1}^{r-1}z_k|} \cdot \prod _{k=1}^{r-1} z_k \cdot \Omega _A \cdot (z_r-z'_r).
		\end{eqnarray*}
		By (\ref{eq3.4}), $\Omega _A \cdot ( z_r -z'_r)=2 \cdot \Omega _A \cdot z_r$, which implies
		\begin{eqnarray*}
			\Omega _A  \cdot \prod _{k=1}^{r}(z_k-z'_k)&=& 2^{r-1} (-1)^{|\Omega _A ||\prod _{k=1}^{r-1}z_k|} \cdot \prod _{k=1}^{r-1}z_k \cdot (2)\Omega _A \cdot z_r\\
			&=& 2^r  \cdot  \Omega _A  \cdot \prod _{k=1}^{r}z_k.
		\end{eqnarray*}
	\end{proof}
	\begin{proof}[ Proof of Theorem \ref{th5.1}]
Suppose $L(\Lambda V, {\mathcal B})=r$. Using the notations introduced before as well as the cocycle $\Omega$ defined in Section \ref{class}, we first construct a cocycle in $(\ker \mu_{\Lambda W _ {\mathcal B}})^{m+n+r}$ whose cohomology class is non-zero. This will permit us to see that $\HTC(\Lambda W_{\mathcal B}) \geq m+n+r$. Since $L(\Lambda V,{\mathcal B})=r$ there exists $r$ cohomology classes, $[z_1] , \cdots , [z_r] \in H _{odd, *}(A)$ such that $[z_1]  \cdots  [z_r] \neq 0.$ Recall that the bigradation of $H(A)$ comes from the splitting of $d$ as the sum of the operators 
$$d_{p,q}: M_{p} (A) \otimes \Lambda ^{q}Y \rightarrow  M_{p+2} (A) \otimes \Lambda ^{q-1}Y.
$$
We can then assume that each cocycle $z_i \in M_{2k_i+1, h_i} (X) \otimes \Lambda ^{h_i} Y$, $\forall i =1,\cdots , r$. It follows that the elements $z_1, \cdots ,z_r$ satisfy the conditions of Lemma \ref{lemma02}. In addition, the element $z= \prod _{k=1} ^r z_k$ is also a cocyle in $A$, and $x _{[n]}\cdot y _{[m]}$ is the unique cocycle (up to a scalar) in $A$ representing the fundamental class of $H^*(A)$. Then, by Poincaré duality in $H^*(A),$ there is another cocycle $\hat{z} \in A$ satisfying  $z\cdot \hat{z} =x_{[n]}\cdot y_{[m]}$. Now, using the Lemma \ref{lemma02} we obtain that
\begin{eqnarray*}
\Omega _A \cdot \prod _{k=1}^r (z_k-z'_k) \cdot \hat{z}&=& 2^r \cdot \Omega _A  \cdot \prod _{k=1}^rz_k \cdot \hat{z}\\
&=& 2^r \Omega_A \cdot x_{[n]}\cdot y_{[m]}\\
&=&  2^r \cdot (-1)^{n+m} x'_{[n]} \cdot y'_{[m]} \cdot x_{[n]}\cdot y_{[m]}\\
&\neq& 0.
\end{eqnarray*}
Since $\varphi$ is a quasi-isomorphism, there exists for each $i$, $[\alpha _i] \in H^*(\Lambda W_{\mathcal B}, d)$ such that $[\varphi (\alpha _i)]= [z_i]$, as well as a cocycle $\hat{\alpha} \in \Lambda W$, such that $[\varphi (\hat{\alpha} )]= [\hat{z}]$. We therefore obtain
		\begin{eqnarray*}
			(\varphi \otimes \varphi) ( \Omega  \cdot \prod _{k=1}^r (\alpha _k-\alpha ' _k) \cdot \hat{\alpha})&=& \Omega _A  \cdot \prod _{k=1}^r (z_k-z'_k)\cdot \hat{z}\\
			&\neq & 0.
		\end{eqnarray*}
		Now, we can see clearly that $ \Omega  \cdot \prod _{k=1}^r (\alpha _k-\alpha ' _k) \cdot \hat{\alpha}$ is a cocycle in $(\ker \mu _{\Lambda W _ {\mathcal B}})^{m+n+r}$, satisfying 
		\begin{eqnarray*}
			[\Omega  \cdot \prod _{k=1}^r (\alpha _k-\alpha ' _k) \cdot \hat{\alpha}]&\neq & 0.
		\end{eqnarray*}
		We then have $$ m+n + L(\Lambda V, {\mathcal B}) \leq \HTC(\Lambda W _ {\mathcal B} ,d) \leq \TC(\Lambda W _ {\mathcal B} ,d). $$
		On the other hand it follows from Corollary \ref{corollary3.1}, that
		$$\TC(\Lambda W _{\mathcal B} )\leq \TC(\Lambda V) +n.$$
		We finally obtain
		$$dim (V^{odd}) +L(\Lambda V,{\mathcal B})\leq \TC(\Lambda V).$$
	\end{proof}

\begin{example}{\rm \label{example1}
We consider a Sullivan model of the form 
$$(\Lambda V,d)=(\Lambda (x_1,x_2,y_1,y_2,y_3),d)$$
 where $|x_1|$ and $ |x_2|$ are even and the differential is given by $dx_1=dx_2=0$, $dy_1=x_1^2$, $dy_1=x_1^2$ and $dy_3=x_1x_2$. For instance the homogeneous space $\frac{SU(6)}{SU(3)\times SU(3)}$ admits such a model with $|x_1|=4$ and $|x_2|=6$.\\ 
Since we have an extension $(\Lambda Z,d) \hookrightarrow (\Lambda V,d)$ where $(\Lambda Z,d)=(\Lambda (x_1,x_2, y_1,y_2),d)$ is an $F_0$-model and $\chi_{\pi}(\Lambda V)=-1$ we have by Theorem \ref{th4.2}  $\TC(\Lambda V)\leq 5$. We have $\cat(\Lambda V)= \dim V^{odd}=3$  and we can check that $\zcl_{\mathbb{Q}} (H(\Lambda V))=3$. We now compute $L(\Lambda V, {\mathcal B})$ where $ {\mathcal B}=\{ x_1, x_2\}$. The extension $\Lambda W_ {\mathcal B}$ is given by $\Lambda W_ {\mathcal B}= (\Lambda V\otimes \Lambda U,d)$ where $U=<u_1,u_2>$ with $du_1=x_1^2$  and $du_2=x_2^2$. We consider the quasi-isomorphism $$\varphi : (\Lambda W_{\mathcal B},d)=(\Lambda V \otimes \Lambda U,d)\rightarrow (A, \bar{d})= \left(\frac{\Lambda (x_i)}{(x_i^2)} \otimes \Lambda (y_1, y_2,y_3), \bar{d}\right)$$
and note that $\bar{d} y_1= \bar{d} y_2=0$ and $\bar{d} y_3=x_1x_2$. The elements of $H_{odd, *}(A)$ corresponds to the cocycles of the form 
\begin{center}
$x_1 \gamma, \quad x_2  \gamma,  \quad x_1 y_3 \gamma, \quad  x_2 y_3 \gamma$ where $\gamma \in \Lambda(y_1,y_2)$
\end{center}
The product of $z_1=x_1$ and $z_2=x_2 y_3$ gives a non-trivial class of $H(A)$ so that $L(\Lambda V, {\mathcal B}) \geq 2$. Since $\cat (\Lambda V)=\dim V^{odd}=3$ we can then conclude by Theorem \ref{th5.1} that $\TC(\Lambda V) \geq 5$ and therefore $\TC(\Lambda V) =\dim V= 5.$ In particular, as indicated in the introduction $\TC_0(\frac{SU(6)}{SU(3)\times SU(3)} )=5$.
}
\end{example}
\begin{remark} \label{ramark1} 
{\rm In the example above we clearly have $$A=\hat{A} \otimes \left(\Lambda(y_1, y_2),0) \right) \quad \text{ where } \hat{A}=\left(\frac{\Lambda (x_i)}{(x_i^2)} \otimes \Lambda y_3, \bar{d}\right)$$  
and $H_{odd, *}(A)\cong H_{odd, *}(\hat{A})\otimes \Lambda(y_1,y_2)$. Consequently the length $L(\Lambda V, {\mathcal B})$ can be calculated  by considering only the element of $H_{odd, *}(\hat{A} )$.}
\end{remark}
\begin{example}
	{\rm
 We consider $(\Lambda V,d)$ given by
 $$ \Lambda V=\Lambda (x_1,x_2,x_3,x_4, y_1,y_2,y_3,y_4,y_5)$$
with $dy_i=x_i^2$ for $i=1, \cdots, 4$ and $dy_5=x_1x_2-x_3x_4.$\\
 As in Example \ref{example1} we can write $\Lambda V$ as an extension $\Lambda Z \otimes  \Lambda y_5$ where $\Lambda Z$ is an $F_0$-model and we have $\TC(\Lambda V) \leq \dim V= 9.$\\
Considering the extension $\Lambda W_{\mathcal B}=\Lambda V \otimes \Lambda U$ associated to ${\mathcal B}=\{x_1, \cdots, x_4\}$ and the observation made in Remark \ref{ramark1} we calculate $L (\Lambda V, {\mathcal B})$ by considering the algebra $\hat{A}= \left(\frac{\Lambda ( x_i)}{(x_i^2)} \otimes \Lambda y_5, \bar{d}\right)$ where $\bar{d} y_5=x_1x_2-x_3x_4$. Note that the fundamental class of the algebra is given by the cocycle $x_1x_2x_3x_4y_5$. We can see that the generators of $H_{odd,*}(\hat{A})$ correspond to the cocycles 
$$x_1,\cdots , x_4, \quad x_1x_2x_3y_5 ,\quad x_1x_2x_4y_5, \quad x_1x_3x_4y_5  , \quad x_2x_3x_4y_5.$$
The maximal non-trivial products of the corresponding cohomology classes are obtained from either products of the form $x_ix_j$ with $i\neq j$ or products of two cocycles resulting in the fundamental cocycle $x_1x_2x_3x_4y_5$. We then obtain that $L(\Lambda V,{\mathcal B})=2$ and the inequality of Theorem \ref{th5.1} gives us $\TC(\Lambda V) \geq 7$, while as mentioned above $\TC(\Lambda V) \leq 9$.

It then appears that the approach through Theorem \ref{th5.1} is not sufficient to completely determine the topological complexity of this example. However, the more specific calculations we do in the next section will permit us to see that in this particular case the actual value is $9$ (see  Theorem \ref{th5.3}). }
\end{example}
\subsection{Special families}
In this section, we are going to determine the rational topological complexity of some families of spaces each of which can be seen as the elliptic extension of a (non-necessarily elliptic) pure coformal space. More exactly we consider elliptic spaces of the form $\Lambda W=(\Lambda(x_i, u_i, y_j),d)$, $du_i=x_i^2$ and $dy_j \in \Lambda (x_i)$ where $i=1, \cdots , n$ and $j=1,\cdots ,m$. Such a space can be seen as the elliptic extension of the (non-necessarly elliptic) model $\Lambda(x_i, y_j)$.
Since the differential is quadratic and $du_i=x_i^2$, through a change of variables we can always suppose that $dy_j = \sum \limits _{i<j} \alpha_{i,j} x_ix_j, \quad \alpha_{i,j} \in \mathbb{Q}$. We will use the construction of the cocycle $\Omega$ as well as the quasi-isomorphism
$$ \varphi: (\Lambda(x_i, u_i, y_j),d) \twoheadrightarrow (A,d)=\left(\frac{\Lambda (x_i)}{(x_i^2)} \otimes \Lambda (y_j), \bar{d}\right),$$
given by $\varphi (x_i)=x_i$, $\varphi (y_j)=y_j$ and $\varphi (u_i)=0$.\\
Recall that $\Lambda W'$ (resp. $A'$) denotes a second copy of $\Lambda W$ (resp. $A$).
\subsubsection{The family $\Lambda W=\Lambda(x_i, u_i,y)$ (case $m=1$)}

As described in Section \ref{class}, let consider the cocycle $\Omega$ associated to the particular case $\Lambda W=\Lambda(x_i, u_i,y)$. In order to find a lower bound for $\TC(\Lambda W)$, we are going to construct a cocycle $\beta \in (\ker \mu _{ \Lambda W})^{n}$ such that
$\Omega \cdot \beta$ represents exactly the fundamental class of $\Lambda W \otimes \Lambda W$. Through the following lemma we assert the existence of such a cocyle $\beta$.
\begin{lemma} \label{lemma5.4}
Suppose that $\Lambda W=\Lambda(x_i,u_i,y)$ where $1\leq i \leq n$. There exists $\gamma \in (\ker \mu_{\Lambda W})^n$ such that $\beta:=\prod \limits_{i=1}^n(x_i-x_i') \cdot y'-\gamma \in (\ker \mu_{\Lambda W})^n$ is a cocycle and $$(\varphi \otimes \varphi)(\gamma)=-\frac{1}{2} \sum \limits _{l=1}^n \prod \limits _{\stackrel{i=1}{i\neq l} }^n(x_i-x_i')(y-y').$$
\end{lemma}
\begin{proof}
	
	 In what follows we use the following notations
	\begin{itemize}
		\item[•] $\pi _{\langle 0 \rangle}= \prod \limits _{ k=1}^n(x_k-x'_k)$ .
		\item[•] $\pi_{\langle l \rangle}= \prod \limits_{ k\neq l}^n(x_k-x'_k)$, for all $l=1, \cdots , n$.
		\item[•] $\pi_{\langle i,j \rangle}= \prod \limits _{ k\neq i,j}^n(x_k-x'_k)$, for all $i,j=1, \cdots , n.$
		\item[•] $ \pi _{\langle i,j,l \rangle}= \prod \limits_{k\neq i,j,l}(x_k-x'_k)$, for all $i,j,l=1, \cdots , n.$ 
	\end{itemize}
Let consider the element $\theta \in (\ker \mu _{\Lambda W})^n$ given by 
$$ \theta = \pi _{\langle 0 \rangle} \cdot y' +\sum _{l=1}^n \frac{1}{2} \pi _{\langle l \rangle} \cdot (y-y').$$
Recall that $d y'$ is given by 
$$ d(y')= \sum _{i<j} \alpha_{i,j} x'_ix'_j, \quad \alpha_{i,j} \in \mathbb{Q}.$$
Then, we have 
\begin{eqnarray*}
	d\theta &=&  \pi _{\langle 0 \rangle} \cdot \sum _{i<j} \alpha_{i,j} x'_ix'_j+ \sum ^n _{l=1} \frac{1}{2} x'_l \cdot \pi _{\langle l \rangle}  \sum _{i<j} \alpha_{i,j} (x_ix_j-x'_ix'_j) \\
	&=& \sum _{i<j} \alpha_{i,j} \pi _{\langle i,j \rangle} (x_i-x'_i)(x_j-x'_j) x'_ix'_j + \sum _{i<j} \alpha_{i,j} \sum ^n _{l=1} \frac{1}{2} x'_l \cdot \pi _{\langle l \rangle}(x_ix_j-x'_ix'_j).   
\end{eqnarray*} 
Using the following identity in the first sum 
\begin{eqnarray*}
	(x_i-x'_i)(x_j-x'_j) x'_ix'_j &=& \frac{1}{2}x'_i(x_j-x'_j)(x'_ix'_j- x_ix_j) +\frac{1}{2}x'_j(x_i-x'_i)(x'_ix'_j- x_ix_j)\\
	&+&\frac{1}{2} x'_i(x_j^2-{x'_j}^2)(x_i-x'_i) + \frac{1}{2} x'_j(x_i^2-{x'_i}^2)(x_j-x'_j) \\
	&+& \frac{1}{2} {x'_i}^2(x_j-x'_j)^2 + \frac{1}{2} {x'_j}^2(x_i-x'_i)^2,
\end{eqnarray*}
we have 
\begin{eqnarray*}
	d \theta &=& \sum _{i<j} \alpha_{i,j} \pi _{\langle i,j \rangle}\,\{ \,
	\frac{1}{2}x'_i(x_j-x'_j)(x'_ix'_j- x_ix_j) +\frac{1}{2}x'_j(x_i-x'_i)(x'_ix'_j- x_ix_j)\\
	&+& \frac{1}{2} x'_i(x_j^2-{x'_j}^2)(x_i-x'_i) + \frac{1}{2} x'_j(x_i^2-{x'}_i^2)(x_j-x'_j) \\
	&+& \frac{1}{2} {x'_i}^2(x_j-x'_j)^2+ \frac{1}{2} {x'_j}^2(x_i-x'_i)^2\,
	\}\\
	&+&  \sum _{i<j} \alpha_{i,j} \frac{1}{2} x'_i\pi _{\langle i \rangle} (x_ix_j-x'_ix'_j) + \sum _{i<j} \alpha_{i,j} \frac{1}{2} x'_j \cdot \pi _{\langle j \rangle} (x_ix_j-x'_ix'_j)\\
	&+& \sum _{i<j} \alpha_{i,j} \sum _{l\neq i,j} \frac{1}{2} x'_l \cdot \pi _{\langle l \rangle}(x_ix_j-x'_ix'_j).  
\end{eqnarray*} 

After reduction we obtain 	
\begin{eqnarray} \label{eq} 
d(\theta)&=&  \sum _{i<j} \alpha_{i,j} \pi _{\langle i,j \rangle}\,\{\, \frac{1}{2} x'_i(x_j^2-{x'_j}^2)(x_i-x'_i) + \frac{1}{2} x'_j(x_i^2-{x'_i}^2)(x_j-x'_j) \nonumber\\
&+& \frac{1}{2} {x'_i}^2(x_j-x'_j)^2 
+ \frac{1}{2} {x'_j}^2(x_i-x'_i)^2
\, \} \nonumber\\
&+& \sum _{i<j} \alpha_{i,j} \sum _{l\neq i,j} \frac{1}{2} x'_l\pi _{\langle l \rangle} (x_ix_j-x'_ix'_j). 
\end{eqnarray}
For $l\neq i,j$ we can write $\pi _{\langle l \rangle}= \pi_{\langle i,j,l \rangle} \cdot (x_i-x_i')(x_j-x_j')$. Using the following decomposition
	\begin{eqnarray*}
	(x_i-x'_i)(x_j-x'_j)(x_ix_j-x'_ix'_j)&=&
	  \frac{1}{2}(x_i-x'_i)(x_j-x'_j)\{ (x_i-x'_i)(x_j+x'_j)+(x_j-x'_j)(x_i+ x'_i) \}\\
	&=&\frac{1}{2}(x_i-x'_i)^2(x_j^2-{x'_j}^2) +\frac{1}{2} (x_j-x'_j)^2(x_i^2-{x'_i}^2)
\end{eqnarray*}
in the last sum of (\ref{eq}) we obtain
\begin{eqnarray*}
d\theta &=&\sum _{i<j} \alpha_{i,j} \pi _{\langle i,j \rangle}\, \{ \,\frac{1}{2} x'_i(x_j^2-{x'_j}^2)(x_i-x'_i)+\frac{1}{2} x'_j(x_i^2-{x'_i}^2)(x_j-x'_j) \\
&+&\frac{1}{2} {x'_i}^2(x_j-x'_j)^2 +\frac{1}{2} {x'_j}^2(x_i-x'_i)^2
\,\}\\
&+& \sum _{i<j} \alpha_{i,j} \sum _{l\neq i,j} \frac{1}{2} x'_l\pi _{\langle i,j,l \rangle}\{  \frac{1}{2}(x_i-x'_i)^2(x_j^2-{x'_j}^2) +\frac{1}{2} (x_j-x'_j)^2(x_i^2-{x'_i}^2)\}.	
\end{eqnarray*}
Since $x_i^2=du_i$, we have $d(\theta)=d(\hat{\theta})$ where
\begin{eqnarray*}
	\hat{\theta}&=& \sum _{i<j} \alpha_{i,j} \pi _{\langle i,j \rangle} \,\{\,\frac{1}{2} x'_i(u_j-u'_j)(x_i-x'_i)+ \frac{1}{2} x'_j(u_i-u'_i)(x_j-x'_j)\\
	& +& \frac{1}{2} u'_i(x_j-x'_j)^2 + \frac{1}{2} u'_j(x_i-x'_i)^2 \,\}\\
	&+& \sum _{i<j} \alpha_{i,j} \sum _{l\neq i,j} \frac{1}{2} x'_l \cdot \pi _{\langle i,j,l \rangle}\{  \frac{1}{2}(x_i-x'_i)^2(u_j-u'_j) +\frac{1}{2} (x_j-x'_j)^2(u_i-u'_i)\}	
\end{eqnarray*}
Note that $\hat{\theta} \in (\ker \mu_{\Lambda W})^n$ and $(\varphi \otimes \varphi)(\hat{\theta})=0$. By setting $\gamma=\hat{\theta}- \sum _{l=1}^n \frac{1}{2} \pi _{\langle l \rangle} \cdot (y-y') $ and 
$\beta= \pi _{\langle 0 \rangle}\cdot y'-\gamma$ we have $\beta= \theta- \hat{\theta}\in (\ker \mu_{\Lambda W})^n$, $d(\beta)=0$ and remark that $\gamma$ satisfies $(\varphi \otimes \varphi)(\gamma)=-\sum \limits _{l=1}^n \frac{1}{2} \pi _{\langle l \rangle} \cdot (y-y')$.
\end{proof}
As we shall show in the following theorem, the elements $\gamma$ and $\beta$ are the main ingredients needed to determine $\TC(\Lambda W)$. 
\begin{theorem} \label{th5.3} For any elliptic extension $\Lambda W= \Lambda(x_i, u_i,y)$ we have $$\TC(\Lambda W)=\dim W$$
\end{theorem}
\begin{proof}
By Lemma \ref{lemma5.4} there exists $\gamma \in (\ker \mu_{\Lambda W})^n $ satisfying $$(\varphi \otimes \varphi)(\gamma)=-\sum _{l=1}^n \frac{1}{2} \pi _{\langle l \rangle} \cdot (y-y')$$
such that $\beta=\pi _{\langle 0\rangle} \cdot y'-\gamma \in (\ker \mu_{\Lambda W})^n$ is a cocycle. From the construction of $\Omega$ (see Section \ref{class}), we have $$(\varphi \otimes \varphi)(\Omega)=(-1)^n\prod_{k=1}^n (x_k-x'_k)(y-y')=(-1)^n\pi _{\langle 0 \rangle}(x_k-x'_k)(y-y').$$ 
It follows that
\begin{eqnarray*}
	[(\varphi \otimes \varphi)(\Omega \cdot \beta) ]&=& [(\varphi \otimes \varphi)(\Omega \cdot(\pi _{\langle 0 \rangle}  \cdot y'-\gamma))]\\
	&=&(-1)^n\prod_{k=1}^n (x_k-x'_k)(y-y') \cdot(\pi _{\langle 0 \rangle} \cdot y'- (\varphi \otimes \varphi)(\gamma)).
\end{eqnarray*}
Since $(\varphi \otimes \varphi)(\gamma)=-\sum \limits _{l=1}^n \frac{1}{2} \pi _{\langle l \rangle} \cdot (y-y')$ we deduce that 
\begin{eqnarray*}
		[(\varphi \otimes \varphi)(\Omega \cdot \beta) ]&=&
		(-1)^n[\prod_{k=1}^n (x_k-x'_k)(y-y') \cdot \prod _{k=1}^n(x_k-x'_k)y'+\sum _{l=1}^n \frac{1}{2} \pi _{\langle l \rangle} \cdot (y-y')]\\
		&=&(-1)^n[\prod_{k=1}^n (x_k-x'_k)(y-y') \cdot \prod _{k=1}^n(x_k-x'_k)y']\\
	&=&2^n [x_{[n]}yx'_{[n]}y'].
\end{eqnarray*}	
Recall that $x_{[n]}y$ and $x'_{[n]}y'$ represent respectively the fundamental classes of $A$ and $A'$ (see Section \ref{class}). Moreover, $x_{[n]}yx'_{[n]}y'$ represents the fundamental class of $A\otimes A'$. Then $\Omega \cdot \beta$ is a cocycle in $(\ker \mu _ {\Lambda W})^{2n+1}$, representing a non-zero cohomology class. As a result, we conclude that $\TC(\Lambda W) \geq 2n+1$. The other inequality ($\TC(\Lambda W)\leq 2n+1)$ follows from Proposition \ref{pr5.1} and finally  $\TC (\Lambda W)=\dim W= 2n+1$.
\end{proof} 
\subsubsection{A special case where $m>1$}
We now consider a model $\Lambda W= \Lambda(x_i,u_i,y_j)$ with $m>1$, and we assume that satisfies the following conditions
	\begin{equation} \label{eq4}
	\begin{cases} 
			dy_1 \in \Lambda (x_1, \cdots , x_{n-1})\\
			d y_j = x_n \cdot \underset{i}{\sum } \alpha^j _i x_i, \quad \alpha ^i _j \in \mathbb{Q} \text{ for all }  j \geq 2.
		\end{cases}
	\end{equation}   
We will see that $\TC(\Lambda W)= \dim W$. 

First we consider the subalgebra $\Lambda \tilde{W}=\Lambda(x_1,\cdots,x_{n-1},u_1,\cdots,u_{n-1},y)$ of $\Lambda W$. We have $\ker \mu _{\Lambda \tilde{W}} \subset \ker \mu _{\Lambda W}$. Applying Lemma \ref{lemma5.4} to $\Lambda \tilde{W}$, there exists $\gamma_1 \in (\ker \mu_{\Lambda \tilde{W}})^{n-1} \subset (\ker \mu_{\Lambda  W})^{n-1}$ such that $\beta_1=\prod \limits_{i=1}^{n-1}(x_i-x_i')y_1'-\gamma_1 \in (\ker \mu_{\Lambda W})^{n-1}$ is a cocycle and $(\varphi \otimes \varphi)(\gamma_1)= -\frac{1}{2} \sum \limits ^{n-1} _{l=1}  \prod \limits_{\stackrel{i=1}{i\neq l} }^{n-1}(x_i-x'_i)(y_1-y'_1).$

Secondly, it is clear that $x_n\cdot y_2 \cdots y_m$ is a cocycle in the algebra $A$. Since $\varphi$ is a quasi-isomorphism there exists $\varepsilon \in \ker \varphi$ such that $d(x_n \cdot y_2 \cdots y_m -\varepsilon )=0.$\\
 Finally considering the cocycle $\Omega \in (\ker \mu_{ \Lambda W})^{n+m}$ for the algebra $\Lambda W$, we construct the cocycle
	$$\alpha:= \Omega \cdot \beta_1 \cdot (x_n \cdot y_2 \cdots y_m -\varepsilon - x'_n \cdot y'_2 \cdots y'_m + \varepsilon ') \in (\ker \mu _{\Lambda W})^{2n+m}.$$
	A calculation shows that
	\begin{eqnarray*} \label{*}
		[(\varphi \otimes \varphi) (\alpha)]&=&[ \prod ^{n-1}_ {i=1} (x_i-x'_i)^2\cdot  \prod ^m_ {j=1} (y_j-y'_j) \cdot  (x_n-x_n')\cdot y'_1 \cdot (x_n y_2 \cdots y_m- x'_n y'_2 \cdots y'_m)]\\
		&=& \pm(-2)^{n-1} [\prod ^{n-1}_ {i=1}x_i x'_i \cdot (x_n-x'_n)\cdot y_1 y'_1 \prod _{j=2}^m(y_j -y'_j) \cdot (x_n  y_2 \cdots y_m- x'_n y'_2 \cdots y'_m)]\\
		&=& \pm(-2)^{n-1}[ x_{[n]} x_{[n]}'\cdot y_1 y'_1 \prod _{j=2}^m(y_j -y'_j) \cdot (y_2 \cdots y_m+y'_2 \cdots y'_m)]\\
		&=&\pm 2^n [x_{[n]}y_{[m]}\cdot x'_{[n]}y'_{[m]}]\\
		&\neq &0. 
	\end{eqnarray*}
	In fact, the last equality is gotten from
	\begin{eqnarray*}
		\prod_{j=2}^m(y_j-y'_j)(y_2 \cdots y_m+ y'_2 \cdots y'_m)&=& y_2\cdots y_m \cdot y'_2 \cdots y'_m +(-1)^{m-1} y'_2 \cdots y'_m \cdot y_2\cdots y_m\\
		&=&(1 +(-1)^{m-1}\cdot(-1)^{|y'_2 \cdots y'_m| |y_2\cdots y_m|} )y_2\cdots y_m \cdot y'_2 \cdots y'_m\\
		&=&(1 +(-1)^{m-1}\cdot(-1)^{(m-1)(m-1)})y_2\cdots y_m \cdot y'_2 \cdots y'_m \\
		&=&(1 +(-1)^{(m-1)m})y_2\cdots y_m \cdot y'_2 \cdots y'_m\\
		&=& 2\cdot y_2\cdots y_m \cdot y'_2 \cdots y'_m.
	\end{eqnarray*}
	As a result $2n+m \leq \TC(\Lambda W)$. Consequently, Proposition \ref{pr5.1} implies  $\TC(\Lambda W)=2n+m$.\\
	Let us finish with the following observation
	\begin{itemize}
	\item For  $n=2$: We can always suppose (through a change of variables) that there exists an $y_1$ such that $d(y_1)=0$. Then the conditions (\ref{eq4}) are satisfied and it follows that $\TC(\Lambda W)=m+4$.
 \item For $n=3$: If there exists an element $y_1$ such that $dy_1=\lambda x_1 x_2$ (with $\lambda \neq 0$) then we can suppose that $dy_j \in \Lambda ^+ (x_3)\otimes \Lambda(x_1,x_2)$ for all $j \geq 2$. Consequently, the conditions (\ref{eq4}) are satisfied and we have $\TC(\Lambda W)=m+6$.
\end{itemize}	
As a consequence of this observation and of the techniques used before, if $(\Lambda V,d)=(\Lambda(x_1,\cdots, x_n, y_1,\cdots, y_m),d)$ is an elliptic pure coformal model with either \begin{itemize}
\item $n=2$ and $m>2$\\
or 
\item $n=3$, $m>3$ and $dy_1=\alpha x_1^2 +\beta x_1x_2+ \gamma x_2^2$ where $\beta \neq 0$
\end{itemize} 
then $TC(\Lambda V) \geq \dim V$.
\section*{Acknowledgements}
This work has been partially supported by Portuguese Funds through FCT -- Funda\c c\~ao para a Ci\^encia e a Tecnologia, within the projects UIDB/00013/2020 and UIDP/00013/2020. A portion of this work has been discussed during the BIRS-CMO workshop \textit{Topological Complexity and Motion Planning}, Oaxaca (Mexico), May 2022, and S.H and L.V thank the \textit{Casa Matemática Oaxaca} for its support and hospitality during this workshop. S.H would like to thank the Moroccan center CNRST --Centre National pour la Recherche Scientifique et Technique for providing him with a research scholarship grant number: 7UMI2020.


\begin{thebibliography}{}
	\addcontentsline{toc}{chapter}{Bibliography}
	\bibitem{AJ}{Alexander, J., Jessup, B. \emph{Explicit formulae for the rational LS-category of some homogeneous spaces,}  J. Pure Appl. Algebra 173 (2002), no. 3, 235--244.}
\bibitem{AM}{Amann, M. \emph{Non-formal homogeneous spaces}. Math. Z. 274 (2013), no. 3-4, 1299–1325. 57N65}
\bibitem{carrasquelsurvey}{Carrasquel, J. G. \emph{Rational methods applied to sectional category and topological complexity.} Topological complexity and related topics, 17–40, Contemp. Math. 702, Amer. Math. Soc., Providence, RI, 2018.}
\bibitem{CJC}{Carrasquel-Vera, J. G. \emph{The Ganea conjecture for rational approximations of sectional category.} J. Pure Appl. Algebra 220 (2016), no. 4, 1310–1315.}
\bibitem{ref5}{Carrasquel-Vera, J. G. \emph{The rational sectional category of certain maps}. Ann. Sc. Norm. Super. Pisa Cl. Sci. (5) 17 (2017), no. 2, 805–813.}
\bibitem{CJV}{Carrasquel-Vera, J. G. \emph{Computations in rational sectional category}. Bull. Belg. Math. Soc. Simon Stevin 22 (2015), no. 3, 455–469.}
\bibitem{cjkl}{Carrasquel-Vera, J.G., Kahl, T., Vandembroucq, L. \emph{Rational approximations of sectional category and Poincaré duality.} Proc. Amer. Math. Soc. 144 (2016), no. 2, 909–915.}
\bibitem{CLOT}{Cornea, O., Lupton, G., Oprea, J., Tanr\'e, D. \emph{Lusternik-Schnirelmann category} Mathematical Surveys and Monographs, 103. American Mathematical Society, Providence, RI, 2003.}




\bibitem{FM}{Farber, M. \emph{Topological complexity of motion planning}. Discrete Comput. Geom. 29 (2003), no. 2, 211–221.}
\bibitem{F}{Farber, M. \emph{Invitation to Topological Robotics}. Zurich Lectures in Advanced Mathematics. European Mathematical Society (EMS), Zurich (2008).}
\bibitem{ref0}{Félix, Y., Halperin, S. \emph{Rational LS category and its applications,} Trans. Amer. Math. Soc. 273 (1982), no. 1, 1-38.} 
\bibitem{FHT}{Félix, Y. Halperin, S., Thomas, C. \emph{Rational homotopy theory,} Graduate Texts in Mathematics, vol. 205, Springer-Verlag, New York, 2001.}





\bibitem{FOT}{F\'elix, Y., Oprea, J., Tanré, D. \emph{Algebraic models in geometry.} Oxford Graduate Texts in Mathematics, 17, Oxford University Press, Oxford, 2008.}

\bibitem{FLGPKT}{Fernández Suárez, L., Ghienne, P., Kahl, T., Vandembroucq, L. \emph{Joins of DGA modules and sectional category}. Algebr. Geom. Topol. 6 (2006), 119–144.)}

\bibitem{JMP}{Jessup, B., Murillo, A., Parent, P.-E. \emph{Rational topological complexity}. Algebr. Geom. Topol. 12 (2012), no. 3, 1789–1801.}

\bibitem{J}{Jessup, B. \emph{LS-category and homogeneous spaces}, Journal of Pure and Applied Algebra, 65 (1990), 45-56.}

\bibitem{LM}{Lechuga, L., Murillo, A. \emph{Topological complexity of formal spaces.} Topology and robotics, 105–114, Contemp. Math., 438, Amer. Math. Soc., Providence, RI, 2007.}

\bibitem{LM02}{Lechuga, L., Murillo, A. \emph{A formula for the rational LS-category of certain spaces}. Ann. Inst. Fourier (Grenoble) 52 (2002), no. 5, 1585–1590.}


\bibitem{L.M}{Lechuga, L., Murillo, A. \emph{The fundamental class of a rational space, the graph coloring problem and other classical decision problems.} Bull. Belg. Math. Soc. Simon Stevin 8 (2001), no. 3, 451–467.}

\bibitem{L}{Lupton, G. \emph{The Rational Toomer Invariant and Certain Elliptic Spaces}, Contemporary Mathematics 316, 2004, 135--146.}

\end{thebibliography}
\end{document}